\documentclass[a4paper,10pt]{article}
\usepackage{amssymb,amsmath,amsthm}
\usepackage{fullpage}
\usepackage{hyperref}
\usepackage{eucal}
\usepackage{mathrsfs}
\usepackage{color}
\usepackage{stackrel}
\usepackage{graphicx}

\newcommand{\eg}{\emph{e.g.}}

\newcommand{\Real}{\mathbb{R}}

\newcommand{\der}{\mathrm{d}}

\newtheorem{Theorem}{Theorem}
\newtheorem{Lemma}{Lemma}
\newtheorem{Proposition}{Proposition}

\theoremstyle{definition}

\def\OMIT#1{}
\usepackage[normalem]{ulem}
\definecolor{DarkGreen}{rgb}{0,0.5,0.1} % David

\newcommand\soutD{\bgroup\markoverwith
{\textcolor{DarkGreen}{\rule[.5ex]{2pt}{1pt}}}\ULon}
\newcommand\soutP{\bgroup\markoverwith
{\textcolor{blue}{\rule[.5ex]{2pt}{1pt}}}\ULon}
\newcommand{\Hm}[1]{\leavevmode{\marginpar{\tiny%
$\hbox to 0mm{\hspace*{-0.5mm}$\leftarrow$\hss}%
\vcenter{\vrule depth 0.1mm height 0.1mm width \the\marginparwidth}%
\hbox to
0mm{\hss$\rightarrow$\hspace*{-0.5mm}}$\\\relax\raggedright #1}}}

\begin{document}
% 
%-------%
% TITLE %
%-------%
%------------------------------------------%
%------------------------------------------%

%
%\maketitle
%
%------------------------------------------%
%------------------------------------------%
\title{A reverse Faber--Krahn inequality for the Robin Laplacian with negative boundary parameter: 
small coupling in all dimensions}
\author{Nunzia Gavitone,$^a$ \
David Krej\v{c}i\v{r}{\'\i}k\,$^b$
\ and \ Gloria Paoli\,$^a$}
\date{\small 
\begin{quote}
\emph{
\begin{itemize}
\item[$a)$]  
Dipartimento di Matematica e Applicazioni ``Renato Caccioppoli'', 
Universit\`a degli Studi di Napoli Federico II, 
Via Cintia, Monte S. Angelo, 80126 Napoli, Italy;
nunzia.gavitone@unina.it, gloria.paoli@unina.it.
\item[$b)$] 
Department of Mathematics, Faculty of Nuclear Sciences and 
Physical Engineering, Czech Technical University in Prague, 
Trojanova 13, 12000 Prague 2, Czechia;
david.krejcirik@fjfi.cvut.cz.%
\end{itemize}
}
\end{quote}
20 June 2026
}
\maketitle
%
%------------------------------------------%
%------------------------------------------%

%
\begin{abstract}
\noindent
We establish Bareket's conjecture from 1977 
for convex domains in all dimensions
in the regime of weak boundary coupling. 
In other words, we consider the Laplace operator,
subject to negative boundary conditions, 
and show that the ball maximises the first eigenvalue
among all bounded convex domains of fixed volume,
provided that the boundary parameter is sufficiently close to zero. 
The smallness depends on the volume and dimension only.
The proof relies on a comparison with spherical shells 
with combined Neumann--Robin boundary conditions 
obtained via the method of parallel coordinates,
which we manage to extend to all dimensions,
and on a careful analysis of the corresponding radial problem.
%
%\bigskip
%\begin{itemize}
%\item[\textbf{Keywords:}]
%\item[\textbf{MSC (2010):}]
%\end{itemize}
%
\end{abstract}
%
%------------------------------------------%
%------------------------------------------%
 
%---------------------%
\section{Introduction}
%---------------------%
%
The circular drum produces the lowest fundamental tone. Or, more generally, among all domains of given volume, the ball minimises the first eigenvalue of the Dirichlet Laplacian.
In a quantum-mechanical language,
the ground-state energy of an electron 
constrained to a nanostructure of given amount of material 
is minimised by the circular geometry.
This is certainly the most famous spectral shape optimisation result,
conjectured by Lord Rayleigh in 1877
and proved independently by Faber and Krahn 
almost 50 years later.  

What happens for other boundary conditions?
Is the ball still the optimal geometry?
The physical relevance as well as mathematical curiosity 
motivates the boundary value problem
\begin{equation}\label{problem}
\left\{
\begin{aligned}
  -\Delta u &=\lambda u 
  && \mbox{in} \quad \Omega \,, \\
  \frac{\partial u}{\partial n} + \alpha \;\! u &=0  
  && \mbox{on} \quad \partial\Omega \,, 
\end{aligned}  
\right.
\end{equation}
where $\alpha \in \Real$ and~$n$ is the outward unit normal 
to a bounded Lipschitz domain $\Omega \subset \Real^d$ with $d \geq 1$.
The case of positive (respectively, negative)~$\alpha$ corresponds
to a strongly localised repulsive (respectively, attractive) 
confining potential on the boundary.
The Dirichlet realisation (formally, $\alpha=\infty$) 
corresponds to hard-wall boundary conditions.
In any case, there exists an infinite sequence of 
discrete eigenvalues of~\eqref{problem}
accumulating at infinity. The ground-state energy is 
the lowest eigenvalue that we denote by~$\lambda_1^\alpha(\Omega)$. 

The case of positive~$\alpha$ is well understood. 
In 1986, Bossel~\cite{Bossel_1986} established 
the Faber--Krahn-type inequality 
$\lambda_1^\alpha(\Omega) \geq \lambda_1^\alpha(B_R)$
valid for $\alpha>0$
and planar domains~$\Omega$ of the same area as the disk~$B_R$
(of positive radius~$R$, centred at the origin without loss of generality).   
The extension to higher dimensions was achieved 
by Daners in 2006~\cite{Daners_2006},
following the original idea of Bossel's. 
In 2010,
Bucur and Giacomini~\cite{Bucur-Giacomini_2010}
proposed an alternative proof based on the theory of special functions 
of bounded variation. 
In 2023,
Alvino, Nitsch, Trombetti \cite{Alvino-Nitsch-Trombetti}
developed yet another proof in two dimensions 
based on a Talenti-type comparison principle.

The case of negative~$\alpha$ remains mysterious.
In 1977, Bareket~\cite{Bareket_1977} raised the conjecture 
that the reverse Faber--Krahn-type inequality 
$\lambda_1^\alpha(\Omega) \leq \lambda_1^\alpha(B_R)$
holds for $\alpha<0$
and all planar domains~$\Omega$ of the same area as the disk~$B_R$.  
Moreover, she proved it for domains close to the disk
provided that~$|\alpha|$ is small. 
Higher-dimensional extensions were considered
by Ferone, Nitsch and Trombetti in 2015~\cite{Ferone-Nitsch-Trombetti_2015}.
In the same year, Freitas and one of the present authors~\cite{FK7}
disproved the conjecture in all non-trivial dimensions $d \geq 2$,
by comparing the eigenvalue asymptotics in the ball
and spherical shells in the limit $\alpha \to -\infty$.
Moreover, they proved that Bareket's conjecture actually holds
for all planar domains~$\Omega$
provided that~$|\alpha|$ is small,
with the smallness depending on the area~$|\Omega|$ only.   
It is still believed 
(and supported by numerical experiments~\cite{AFK})
that the ball is the maximiser within the class
of simply connected domains if $d=2$
and convex domains if $d \geq 2$ 
(in dimensions $d\geq 3$, it is not enough to assume that
the domain is simply connected~\cite{Ferone-Nitsch-Trombetti_2016}).
In particular, it was conjectured in~\cite[Conj.~1]{AFK}
that there exists a negative number~$\alpha_0$,
depending solely on the volume~$|\Omega|$ and dimension~$d$, 
that the reverse Faber--Krahn-type inequality 
$\lambda_1^\alpha(\Omega) \leq \lambda_1^\alpha(B_R)$
holds for all domains $\Omega\subset\Real^d$ 
and every $\alpha \in [\alpha_0,0]$.
 
In this paper, we do not prove Bareket's conjecture in its full generality.
However, we support its validity by extending the positive result of~\cite{FK7}
to higher dimensions. 
In particular, we establish~\cite[Conj.~1]{AFK}
in the special case of convex domains, in all dimensions.
\begin{Theorem}\label{Thm.main}
Let $\Omega \subset \Real^d$ be a bounded convex domain with $d \geq 1$. 
There exists a negative number~$\alpha_0$,
depending only on the volume~$|\Omega|$ and dimension~$d$,
such that 
\begin{equation} 
  \forall \alpha \in [\alpha_0,0] \,, \qquad
  \lambda_1^\alpha(\Omega) \leq \lambda_1^\alpha(B_R)
  \,,
\end{equation}
where~$B_R$ is the ball of the same volume as~$\Omega$.
\end{Theorem}

The validity of this theorem in $d \geq 3$ has been open
at least since 2015, when the planar situation was settled in~\cite{FK7}.  
(The case $d=1$ is trivial.)
The problem was the unavailability of the method of parallel coordinates
(as originally developed by Payne and Weinberger~\cite{Payne-Weinberger_1961}
in the Dirichlet case) in higher dimensions. 
In fact, it is explicitly stated in~\cite{FK5} 
that the proof of the results of~\cite{Payne-Weinberger_1961} 
``does not seem to have a straightforward extension to higher dimensions'' 
(see~\cite[p.~8]{FK5}).
In this paper, we manage to establish this higher-dimensional extension
by being inspired by the isoperimetric approach of
Bucur, Ferone, Nitsch and Trombetti from 2019
\cite{Bucur-Ferone-Nitsch-Trombetti_2019}
as well as by our recent paper for the torsional rigidity~\cite{GKP}.
More specifically, we prove the following lemma
(where we denote by $|\Omega|$ and~$|\partial\Omega|$
the $d$-dimensional Lebesgue measure of~$\Omega$ 
and the $(d-1)$-dimensional Hausdorff measure of 
the boundary~$\partial\Omega$, respectively). 

\begin{Lemma}\label{Lem.PW}
For any bounded convex domain $\Omega \subset \Real^d$ 
with $d \geq 2$, one has
\begin{equation}\label{eigenvalue_annulus}
  \forall \alpha<0 \,, \qquad
    \lambda_1^{\alpha}(\Omega )\leq \mu_1^{\alpha}(A_{R_1,R_2}),
\end{equation}
where $A_{R_1,R_2} := B_{R_2} \setminus \overline{B_{R_1}}$ 
is the spherical shell such that 
$|\partial B_{R_2}|=|\partial\Omega|$ and $|A_{R_1,R_2}|=|\Omega|$.
Here $\mu_1^{\alpha}(A_{R_1,R_2})$ denotes the lowest eigenvalue 
of the combined Robin--Neumann problem
\begin{equation}\label{problem.annulus}
\left\{
\begin{aligned}
  -\Delta u &=\mu u 
  && \mbox{in} \quad A_{R_1,R_2} \,, \\
  \frac{\partial u}{\partial n} + \alpha \;\! u &=0  
  && \mbox{on} \quad \partial B_{R_2} \,, \\
  \frac{\partial u}{\partial n}  &=0  
  && \mbox{on} \quad \partial B_{R_1} \,.
\end{aligned}  
\right.
\end{equation}
\end{Lemma}

If $d=2$ and $d=3$, the lemma is due to~\cite{FK7} and~\cite{Vikulova_2022},
respectively.
The main technical achievement of the present paper is the extension 
of the lemma to higher dimensions.
It easily follows from Lemma~\ref{Lem.PW} 
and the eigenvalue asymptotics 
\begin{equation}\label{analytic}
\begin{aligned}
  \mu_1^\alpha(A_{R_1,R_2}) 
  &= \frac{|\partial B_{R_2}|}{|A_{R_1,R_2}|}
  \, \alpha + O(\alpha^2) \,,
  \\
  \lambda_1^\alpha(B_R) 
  &= \frac{|\partial B_{R}|}{|B_{R}|}
  \, \alpha + O(\alpha^2) \,,
\end{aligned}
\qquad \mbox{as} \quad \alpha \to 0,
\end{equation}
that Theorem~\ref{Thm.main} holds
with the critical constant~$\alpha_0$ which additionally depends
on the perimeter~$|\partial\Omega|$.
A more refined analysis of the radially symmetric situations 
is needed to get rid of this extra dependence. 
In the planar situation of~\cite{FK7}, 
this was achieved by analysing the explicit solutions 
in terms of Bessel functions.  
In the present paper, we provide a methodologically new approach.
What is more, it is our belief that the validity of Lemma~\ref{Lem.PW}
and our proof of it will be appreciated by the community
in other circumstances.

The organisation of this paper is as follow.
In Section~\ref{Sec.pre} we recall basic facts about convex sets
and derive auxiliary results about the radial problems in balls 
and spherical shells.
The proofs of Theorem~\ref{Thm.main} and Lemma~\ref{Lem.PW}
are presented in Section~\ref{Sec.proofs}. 

%-----------------------%
\section{Preliminaries}\label{Sec.pre}
%-----------------------%
%
 
\subsection{Elements of convex analysis}
Let us start by recalling some properties of convex bodies, 
which will be useful in the sequel (we refer to~\cite{schneider}
for more details). 
Let $\emptyset\neq \Omega\subseteq\mathbb{R}^d$ with $d \geq 2$
be a bounded convex domain. 
We define the \emph{outer parallel body} of~$\Omega$ 
at distance $t \geq 0$ as the  Minkowski sum
$$ 
\Omega+t B_1=\{ x+ t y\in\mathbb{R}^d:\  x\in \Omega,\;y\in B_1 \}.
$$ 
The Steiner formula asserts that
\begin{equation}\label{general_steiner}
|\Omega+ t B_1|=\sum_{i=0}^{d}\binom{d}{i} \, W_i(\Omega) \, t^i
\,, 
\end{equation}
where the coefficients $W_i(\Omega)$ are usually called \emph{quermassintegrals}.
For smooth~$K$, one has $W_2(\Omega) = \frac{1}{d} \int_{\partial K} H d\mathcal{H}^{D-1}$,
where $H$ denotes the mean curvature of~$\partial \Omega$. The Aleksandrov--Fenchel inequalities state 
\begin{equation}\label{aleksandrov-fenchel}
	\left(\dfrac{W_j(\Omega)}{|B_1|}\right)^{\frac{1}{d-j}}
	\geq \left(\dfrac{W_i(\Omega)}{|B_1|}\right)^{\frac{1}{d-i}}
	\qquad\mbox{for}\qquad
	0\leq i<j<d \,,
\end{equation}
with equality if, and only if, $\Omega$~is a ball.  
In the case  $i=0$ and $j=1$, 
we obtain the classical isoperimetric inequality
\begin{equation*}%\label{classical_iso}
|\partial \Omega|^{\frac{d}{d-1}}
\geq d^{\frac{d}{d-1}} \, |B_1|^{\frac{1}{d-1}} \, |\Omega|
\,. 
\end{equation*}

Let $\rho(x) := \inf_{y \in \partial \Omega} |x-y|$
denote the distance from the boundary of~$\Omega$ 
to a point $x \in \Omega$.
The maximal distance from the boundary, 
$r_{\Omega} := \sup_{x \in \Omega} \rho(x)$,
is called the inradius of~$\Omega$.
For every $t \in[0,r_{\Omega}]$, we consider the super-level sets of the distance function 
\begin{equation*}%\label{def_omega_t}
\Omega_{t}:=\{ x\in \Omega:\ \rho(x)>t  \} 
\,,
\end{equation*}
We recall the following lemmas, whose proof can be found, 
for instance, in \cite{Bucur-Ferone-Nitsch-Trombetti_2019,paoli_trani_piscitelli}.

\begin{Lemma}
    Let $\Omega$ be a bounded, convex, open set in $\mathbb{R}^d$. Then, for almost every $t\in(0,r_{\Omega} )$, we have
\begin{equation*}\label{derivative_perimeter}
		-\dfrac{d}{dt}P(\Omega_{t})\geq d(d-1)W_2(\Omega_{t})
	\end{equation*}
	and equality holds if $\Omega$ is a ball.
\end{Lemma}

\begin{Lemma}
\label{lem_der_per_e}
Let $\Omega \subset \Real^d$ with $d \geq 2$
be a bounded convex domain. 
In addition, let $f:[0,+\infty)\to[0,+\infty) $  be a non-increasing $C^1$ function.   %and let $\tilde f:[0,+\infty)\to[0,+\infty) $  a non increasing  $C^1$ function. 
    Define  
	\begin{equation*}
	u(x):=f(\rho(x))
	\qquad \mbox{and} \qquad 
	E_{t}:=\{x\in \Omega\;:\; u(x)<t\}. 
	\end{equation*}
	Then, for a.e. regular value \(t\)
\begin{equation}\label{derivata_composta_super1}
		\dfrac{\der}{\der t}
		|\partial E_{t}|
		\geq d (d-1) \, \dfrac{W_2(E_{t})}{|D u|_{u=t}}
		\,
	\end{equation}
with equality holding if $E_t$ is a ball.
	%and 
		%\begin{equation}\label{derivata_composta_sub1}
	%	\dfrac{d}{dt}P(\tilde E_{0,t})\geq n (n-1)\dfrac{W_2(\tilde E_{0,t})}{|D \tilde u|_{\tilde u=t}}.
	%\end{equation}
\end{Lemma}

\subsection{Spherical shells}
Given two non-negative real numbers $R_2> R_1$, 
let $A_{R_1,R_2}$ be the spherical shell of radii $R_1,R_2$
as defined in the statement of Lemma~\ref{Lem.PW}.  
For any $\alpha<0$, we consider 
the eigenvalue problem~\eqref{problem.annulus}. 
The first eigenvalue admits the variational characterisation
\begin{equation}\label{eigRN}
\mu_1^{\alpha}( A_{R_1,R_2})
=\min_{\substack{w\in W^{1,2}(A_{R_1,R_2})\\ w\not=0}}
\dfrac{\int _{A_{R_1,R_2}}|\nabla w|^2 
+\alpha\int_{\partial B_{R_2}}|w|^2 }{\int_{A_{R_1,R_2}}|w|^2} 
\,.
\end{equation}
By taking a constant trial function in~\eqref{eigRN},
we get  
\begin{equation}\label{upper}
\mu_1^{\alpha}(A_{R_1,R_2})
\le \alpha \, \frac{|\partial B_{R_2}|}{|A_{R_1,R_2}|}
\,.
\end{equation}
In particular, we observe that $\mu_1^{\alpha}(A_{R_1,R_2})$ is negative.
It is well known that the lowest eigenvalue is simple
and that the corresponding eigenfunction~$u$ 
can be chosen to be positive. 
The following proposition is easily established 
(see, \eg, \cite{paoli_trani_piscitelli}).
\begin{Proposition} \label{shell}
Let~$u_1$ be a positive minimiser of~\eqref{eigRN}. 
Then~$u_1$ is radially symmetric, in the sense that 
there exists a smooth  function  $\phi:[R_1,R_2] \to (0,\infty)$
such that $u_1(x)= \phi(|x|)$
for every $|x| \in [r_1,r_2]$. 
Moreover, $\phi$~is strictly increasing.
\end{Proposition}

\subsection{Balls}
We continue to assume that~$\alpha$ is negative.
Let~$B_R$ be a $d$-dimensional ball of positive radius~$R$
centred at the origin of~$\mathbb R^d$. 
Recall that $\lambda_1^\alpha(B_R)$ denotes the lowest eigenvalue 
of the problem~\eqref{problem} with~$\Omega$ being replaced by~$B_R$. 
One has $\lambda_1^\alpha(B_R) = \mu_1^\alpha(A_{0,R})$.
In particular, Proposition~\ref{shell} remains valid 
for $B_R = A_{0,R}$. More specifically, one has 
\begin{equation}\label{problem.ball}
\left\{
\begin{aligned}
- r^{-(d-1)}\left( \phi'(r)r^{d-1}\right)' 
&= \lambda_1^\alpha(B_R) \phi(r) \,,
&& r\in (0,R) \,,\\
\phi'(0) &=0 \,, \\
\phi'(R)+\alpha\phi(R) &=0 \,.
\end{aligned}
\right.
\end{equation}
Recalling the notation of Proposition~\ref{shell},
we have the following comparison results.
\begin{Proposition}\label{estimates_radial}
Let $u$ be a positive eigenfunction 
corresponding to $\lambda_1^\alpha(B_R)$. 
Then 
\begin{itemize}
\item[\emph{(i)}] 
$ |\lambda_1^{\alpha}(B_R)|
\le |\alpha| \, \displaystyle \frac{d}{R} \,
\frac{\phi(R)}{\phi(0)}$,
\item[\emph{(ii)}] 
$
  \phi(R)-\phi(0)
  \le |\lambda_1^{\alpha}(B_R)| \, \phi(R) \, \displaystyle\frac{R^2}{2d}
$.
\end{itemize}
\end{Proposition}
\begin{proof}
From~\eqref{problem.ball}, we deduce, for every $r (0,R]$,
\begin{equation*}%\label{der}
    \phi'(r)=|\lambda_1^{\alpha}(B_R)| \, r^{-(d-1)}
    \int_0^r \phi(t) \, t^{d-1} \, \der t
    \,.
\end{equation*}
Since~$\phi$ is increasing, the desired properties follow at once.
\end{proof}

%----------------%
\section{Proofs}\label{Sec.proofs}
%----------------%
%

\subsection{Proof of Lemma~\ref{Lem.PW}}
First, we establish the Payne--Weinberger-type result  
stated in Lemma~\ref{Lem.PW}.
The Dirichlet and Robin cases for two-dimensional 
(not necessarily convex, it is enough to assume simply-connected) domains
are originally due to~\cite{Payne-Weinberger_1961} and~\cite{FK7},
respectively. 
The three-dimensional case 
(for convex or axiconvex domains)
was settled in~\cite{Vikulova_2022}.
The main methodological novelty of the present paper
is the extension of the Payne--Weinberger result
to higher dimensions $d \geq 4$. 
We are inspired by the approaches of 
\cite{Bucur-Ferone-Nitsch-Trombetti_2019}, 
\cite{paoli_trani_piscitelli} and~\cite{GKP}
developed for other constraints
or different geometric and analytic settings, respectively.  
 
%\begin{proof}[Proof of Lemma~\ref{Lem.PW}] 
Let $\phi$ be the increasing function from Proposition~\ref{shell},
related with the radial positive eigenfunction corresponding to
$ \lambda_1^{\alpha}(A_{R_1,R_2})$. 
For simplicity, we abbreviate $A_{R_1,R_2} =: A$, 
$\phi(R_2)=:\phi_M$ and $\phi(R_1)=:\phi_m$.  
In~$\Omega$, we construct the trial function 
\begin{equation*}
u(x):=
\begin{cases}
G(\rho(x)) & \text{if} \quad  \rho(x)< R_2-R_1 \,, 
\\
\phi_m & \text{if} \quad \rho(x)\geq R_2-R_1 \,,
\end{cases}
\end{equation*}
where~$\rho$ is the distance function to the boundary of~$\Omega$ 
and $G:[0,R_2-R_1] \to [\phi_m,\phi_M]$ is defined by
%\begin{equation*}
    $G(s):=\phi(R_2-s)$.
%\end{equation*}
%
Equivalently
	\begin{equation*}%\label{G_t_max}
		G^{-1}(t) = \int_{t}^{\phi_M}\dfrac{\der\tau}{g(\tau)} 
		\,,
	\end{equation*}
with $g:[\phi_m,\phi_M] \to (0,\infty)$ 
defined by $g(t) := |\phi'|_{\phi=t}$.
We observe that $u$ satisfies the following properties: 
\begin{align*}
u\in W^{1,2}(\Omega), 
&&&
u_m:=\inf_\Omega u \geq \phi_m, 
\\
|\nabla u|_{u=t}=g(t), 
&&&
u_M:=\sup_\Omega u = \phi_M=G(0).
\end{align*}

For $t\in[\phi_m,\phi_M]$ define 
\begin{equation*}
E_t:=\{x\in\Omega:\ u(x)<t\},
\qquad
A_t:=\{x\in A:\ \phi(|x|)<t\},
\qquad
B_t:=A_t\cup \overline B_{R_1}.
\end{equation*}
Thus $B_t$ is a ball of radius $\phi^{-1}(t)$.
Since $E_t$ and $B_t$ are convex sets, Lemma~\ref{lem_der_per_e}
and the Aleksandrov--Fenchel inequality yield, 
for a.e.\ $t\in(u_m,\phi_M)$,
	\begin{equation*}
	\dfrac{\der}{\der t}|\partial  E_{t}|
	\geq d (d-1)\dfrac{W_2( E_{t})}{g(t)}
	\geq d(d-1)d^{-\frac{d-2}{d-1}} |B_1|^{\frac{1}{d-1}}
	\dfrac{
	|\partial  E_{t}|^{\frac{d-2}{d-1}}}{g(t)}
	\,,
	\end{equation*}
whereas equality holds for the balls $B_t$:
\begin{equation*}
\frac{\der}{\der t}|\partial B_t|
= d(d-1)d^{-\frac{d-2}{d-1}} |B_1|^{\frac{1}{d-1}}
\frac{|\partial B_t|^{\frac{d-2}{d-1}}}{g(t)}.
\end{equation*}
	Since, by hypothesis, $|\partial\Omega|=|\partial B_{R_2}|$, 
	using  a comparison type theorem, we obtain 
	\begin{equation}\label{eq:perimeter-comparison-eigen}
	 |\partial  E_{t}|
	 \leq 
	 |\partial B_{t}|
	\end{equation} 
	for a.e.\ $t\in(u_m,\phi_M)$. 
Define $\mu(t):=|E_t|$ and $\eta(t):=|A_t|$.
Using the coarea formula and~\eqref{eq:perimeter-comparison-eigen},
we obtain, for a.e.\ $t\in(u_m,\phi_M)$,
\begin{equation*}
\mu'(t)
=\int_{\{ u=t\}}\frac{1}{|\nabla u(x)|}\, \der \mathcal{H}^{d-1}
= \frac{|\partial E_t|}{g(t)}
\leq \frac{|\partial B_t|}{g(t)}
=\int_{\{ \phi=t\}}\frac{1}{|\nabla\phi(x)|}\, \der\mathcal{H}^{d-1}
=\eta'(t).
\end{equation*}
Moreover, $\mu(\phi_M)=|\Omega|=|A|=\eta(\phi_M)$.
Integrating the last inequality from $u_m$ to $\phi_M$ gives
\begin{equation*}
|\Omega|=\mu(\phi_M)-\mu(u_m)
\leq \eta(\phi_M)-\eta(u_m)=|A|-\eta(u_m).
\end{equation*}
Since $|\Omega|=|A|$, we get $\eta(u_m)=0$, and hence $u_m=\phi_m$.
Integrating instead from an arbitrary $t\in[\phi_m,\phi_M)$ to $\phi_M$, we infer that
\begin{equation*}
|\Omega|-\mu(t)\leq |A|-\eta(t),
\end{equation*}
and therefore $\mu(t)\geq\eta(t)$ for $\phi_m\leq t<\phi_M$.
Using the coarea formula again, we get
\begin{equation}\label{gradient_estimates_e-}
\int_{\Omega} |\nabla u|^2  \, \der x
=  \int_{\phi_m}^{\phi_M} g(t) \, |\partial E_t| \, \der t
\leq \int_{\phi_m}^{\phi_M} g(t) \, |\partial B_t| \, \der t
=\int_{A}|\nabla \phi|^2 \, \der x.
\end{equation}	
Since, by construction, $u(x)=u_M=\phi_M$ on $\partial\Omega$, it holds
\begin{equation}\label{termine_bordo_Me}
\int_{\partial \Omega}u^2\,\der\mathcal{H}^{d-1}
=u_M^2 \, |\partial\Omega|
= \phi_M^2 \, |\partial B_{R_2}|
=\int_{\partial B_{R_2}}\phi^2\, \der\mathcal{H}^{d-1}.
\end{equation}
Consequently,
\begin{equation}\label{L_p_estimates_neg_e-}
\int_{\Omega}u^2 \, \der x 
= u_M^2|\Omega|-\int_{\phi_m}^{\phi_M}2t\mu(t) \, \der t 
\leq    \phi_M^2|A|-\int_{\phi_m}^{\phi_M}2t\eta(t)\, \der t
= \int_{A}\phi^2\, \der x.
\end{equation}
Finally, by~\eqref{gradient_estimates_e-}, \eqref{termine_bordo_Me} and~\eqref{L_p_estimates_neg_e-},
and since the numerator on the shell is negative, we have
\begin{equation*}
\lambda_1^{\alpha}(\Omega)
\leq \frac{\int_{\Omega}|\nabla u|^2\,\der x
+\alpha\int_{\partial \Omega}u^2\,\der\mathcal{H}^{d-1}}
{\int_{\Omega}u^2\,\der x}
\leq \frac{\int_{A}|\nabla\phi|^2\,\der x+\alpha\int_{\partial B_{R_2} }\phi^2\,\der\mathcal{H}^{d-1}}{\int_{A}\phi^2\,\der x}
=\mu_1^{\alpha}(A).
\end{equation*}	
This concludes the proof of Lemma~\ref{Lem.PW}.
\qed
%\end{proof}
%

\subsection{Proof of Theorem~\ref{Thm.main}}
Given any bounded convex domain~$\Omega\subset\mathbb{R}^d$, $d\geq 2$,
let the ball~$B_R$ be such that $|\Omega| = |B_R|$; 
the radius~$R$ depends solely on the volume~$|\Omega|$ and dimension~$d$. 
Let~$A_{R_1,R_2}$ be such that $|A_{R_1,R_2}|=|\Omega|$
and $|\partial B_{R_2}| = |\partial\Omega|$;
the radii $R_1,R_2$ depend on the perimeter $|\partial\Omega|$,
the volume~$|\Omega|$ and dimension~$d$.
The volume constraint gives
\begin{equation}\label{relationship}
  R_2^d - R_1^d = R^d
  \,,
\end{equation}	
from which we express~$R_2$ as a function of~$R_1$.

By Lemma~\ref{Lem.PW}, one has
\begin{equation}\label{crucial}
  \lambda_1^{\alpha}(\Omega) - \lambda_1^{\alpha}(B_R) 
  \leq \mu_1^{\alpha}(A_{R_1,R_2}) - \lambda_1^{\alpha}(B_R) 
  =: \gamma(\alpha,R_1)
  \,. 
\end{equation}	
Here $\gamma(\alpha,R_1)$ also depends on~$|\Omega|$ and~$d$,
but this dependence is not important for our purposes. 
Indeed, our goal is to show that there exists $\alpha_0 < 0$,
depending solely on~$|\Omega|$ and~$d$,
such that $\gamma(\alpha,R_1) \leq 0$ for every $\alpha \in [\alpha_0,0]$.
So the main task is to ensure that the critical value~$\alpha_0$
can actually be made independent of~$R_1$.
We divide the proof into to regimes.

\medskip
\noindent
\fbox{Case $R_1 \geq R$.}
By~\eqref{upper}
and the second asymptotic formula in~\eqref{analytic},
one has 
\begin{equation*} 
\begin{aligned}
  \gamma(\alpha,R_1) \leq 
  \alpha \, \frac{|\partial B_{R_2}| - |\partial B_{R}|}{|B_R|}
  + O(\alpha^2)
  \,,
\end{aligned}  
\end{equation*}	
where the error term $O(\alpha^2)$ depends solely on~$|\Omega|$ and~$d$.
Since $R_2 \geq 2^\frac{1}{d} R$, there clearly exists a negative~$\alpha_1$,
depending solely on~$|\Omega|$ and~$d$,
such that $\gamma(\alpha,R_1) \leq 0$ for every $\alpha \in [\alpha_1,0]$.

It is also clear that this asymptotic argument cannot be used
without a restriction on~$R_1$. 
Indeed, $|\partial B_{R_2}| \to |\partial B_{R}|$ as $R_1 \to 0$,
which makes~$\alpha_1$ necessarily dependent also on~$R_1$
(and thus on~$|\partial\Omega|$). 
A more subtle argument is needed to cover the spherical shells
close to the ball~$B_R$.  

\medskip
\noindent
\fbox{Case $R_1 < R$.}
Recall the notation of Proposition~\ref{shell}.
Let~$u$ be the positive first eigenfunction of~$B_R$
normalised by $u(0)=1$. Then also $\phi(0)=1$.
For $\alpha=0$ the problem~\eqref{problem.annulus}
becomes uniformly Neumann,
for which the first eigenfunction is constant. 
Consequently, using an analytic dependence of~$u$ on~$\alpha$,
there exists a negative~$\alpha^*$, depending only on~$R$ and~$d$, 
such that 
\begin{equation}\label{osc}
1 \le \phi(r) \le 2, \quad \forall r \in[0,R],\,\,\,\forall \alpha \in [\alpha^*,0].
\end{equation}
For the rest of the proof, let us abbreviate 
$\mu_1^\alpha(A_{R_1,R_2}) =: \mu(\alpha,R_1)$
and $\lambda_1^\alpha(B_R) =: \lambda(\alpha)$.
Recall that~$R_2$ depends on~$R_1$ through~\eqref{relationship}.

Now we define a radial test function $w \in W^{1,2}(A_{R_1,R_2})$ 
in the following way: 
\[
w(x):=
\begin{cases}
\phi(|x|) & \mbox{if}\quad  R_1 \le |x| \le R,
\\
\phi(R) & \mbox{if}\quad R\le |x| \le R_2.
\end{cases}
\]
Using this trial function in the variational characterisation~\eqref{eigRN},
we obtain 
\begin{equation}\label{quotient}
    \mu(\alpha,R_1)
    \le
    \frac{\int_{R_1}^{R} \phi'(r)^2 \, r^{d-1}\,\der r
    +\alpha \, R_2^{d-1} \, \phi(R)^2}
    {\int_{R_1}^{R}\phi(r)^2 \, r^{d-1}\,\der r
+\phi(R)^2\int_R^{R_2} r^{d-1}\,\der r}.
\end{equation}
Since~$\phi$ is the first ball eigenfunction, 
we have 
\begin{equation}\label{eq:ball-radial-identity}
\int_0^R \phi'(r)^2 \, r^{d-1}\,\der r
+\alpha \, R^{d-1} \, \phi(R)^2
=\lambda(\alpha) \int_0^R \phi(r)^2 \, r^{d-1}\,\der r.
\end{equation}
Consequently, 
\begin{equation}
     \mu(\alpha,R_1) \le 
     \dfrac{ \lambda(\alpha)\int_0^R \phi(r)^2 \, r^{d-1} \, \der r
     -\int_0^{R_1}|\phi'(r)|^2 \, r^{d-1}\,\der r
     +\alpha \, (R_2^{d-1}-R^{d-1}) \, \phi(R)^2}
     {\int_{R_1}^{R}\phi(r)^2 \, r^{d-1}\,\der r
+\phi(R)^2\int_R^{R_2} r^{d-1}\,\der r}
     \,.
\end{equation}
Subtracting $\lambda(\alpha)$, we obtain
\begin{equation*}%\label{rayleigh}
    \gamma(\alpha,R_1) \leq \frac{N}{D},
\end{equation*}
where 
\begin{align*}%\label{num}
D &:=\int_{R_1}^R \phi(r)^2 \, r^{d-1}\,\der r
    + \phi(R)^2 \int_R^{R_2} r^{d-1}\,\der r,
\\
N&:=-\int_0^{R_1} \phi'(r)^2 \, r^{d-1}\,\der r
+\alpha\, (R_2^{d-1}-R^{d-1}) \, \phi(R)^2 
%\notag\\ &\phantom{{}={}}\ 
\\
& \qquad
+\lambda(\alpha)\left(\int_0^{R_1} \phi(r)^2 \, r^{d-1}\,\der r
-\phi(R)^2\int_R^{R_2} r^{d-1}\,\der r \right).
\end{align*}
It remains to ensure that the numerator~$N$ is non-positive,
under a suitable smallness condition on~$\alpha$. 

First of all, 
we note that the volume constraint~\eqref{relationship} implies 
\begin{equation*}%\label{constraint}
\int_R^{R_2} r^{d-1}\,\der r
=\frac{R_2^d-R^d}{d}
=\frac{R_1^d}{d}
=\int_0^{R_1} r^{d-1}\,\der r.
\end{equation*}
Consequently,
\begin{equation*}%\label{eq:mass-difference}
\int_0^{R_1}\phi(r)^2 \, r^{d-1}\,\der r
-\phi(R)^2 \int_R^{R_2} r^{d-1}\,\der r
=\int_0^{R_1}\bigl(\phi(r)^2-\phi(R)^2\bigr) \, r^{d-1}\,\der r.
\end{equation*}
Second, by~\eqref{relationship}, we have
\begin{equation*} 
R_2^{d-1}-R^{d-1}
=(R^d+R_1^d)^{\frac{d-1}{d}}-R^{d-1}
=R^{d-1}\left(\left(1+\left(\frac{R_1}{R}\right)^d\right)^{\frac{d-1}{d}}-1\right) \ge  \dfrac{(d-1) \, 2^{-1/d} }{d R} \, R_1^d
\,.
\end{equation*}
Consequently, using in addition 
Propositions~\ref{shell} and~\ref{estimates_radial}, 
we get 
\begin{equation*} 
\begin{aligned}
    N  
    &\leq \alpha \, \dfrac{(d-1) \, 2^{-1/d}}{d R} \, R_1^d \, \phi(R)^2
    + |\lambda(\alpha)| \int_0^{R_1}
    \bigl(\phi(R)^2-\phi(r)^2\bigr) \, r^{d-1}\,\der r
    \\
    &\leq \alpha \, \dfrac{(d-1) \, 2^{-1/d}}{d R} \, R_1^d \, \phi(R)^2
    + |\alpha| \, \dfrac{d}{R} \, 
    \dfrac{\phi(R)}{\phi(0)}\dfrac{R_1^d}{d} \, 2\phi(R) \, (\phi(R)-\phi(0))
    \notag\\
    & \leq 
    \alpha \, \dfrac{(d-1) \, 2^{-1/d}}{d R} \, R_1^d \, \phi(R)^2
    + |\alpha|^2 \, \dfrac{\phi(R)^4}{\phi(0)^2} \,  R_1^d 
     \notag\\& =
      |\alpha| \, R_1^d \, \phi(R)^2
     \left(
     -\frac{(d-1) \, 2^{-1/d}}{dR}
     + |\alpha| \, \dfrac{\phi(R)^2}{\phi(0)^2}
     \right) 
     .
\end{aligned}
\end{equation*}
Recalling~\eqref{osc} and the normalisation $\phi(0)=1$, 
it follows that~$N\leq 0$ provided that  
$$
  |\alpha| \leq 
  \min\left\{
  |\alpha^*|,
  \frac{(d-1) \, 2^{-1/d}}{4dR} 
  \right\} =: -\alpha_2.
$$
Obviously, $\alpha_2$ depends solely on~$R$ and~$d$.

\medskip
\noindent
\fbox{Arbitrary $R_1$.}
Combining the previous cases, 
it follows that $\gamma(\alpha,R_1) \leq 0$
whenever $\alpha \in [\alpha_0,0]$
with $\alpha_0 := \max\{\alpha_1,\alpha_2\}$.
This concludes the proof of Theorem~\ref{Thm.main}.
\qed

\subsection*{Acknowledgement}
N.~G.\ and G.~P.\  were partially supported by Gruppo Nazionale per l'Analisi Matematica, la Probabilit?? e le loro Applicazioni
(GNAMPA) of Istituto Nazionale di Alta Matematica (INdAM).
D.K.\ was supported
by the grant no.~26-21940S
of the Czech Science Foundation.

G.P was supported by COST Action 24122 mSPACE, supported by COST (European Cooperation in Science and Technology), www.cost.eu

%\newpage
\vfill  
%--------------%
% BIBLIOGRAPHY %
%--------------%
%
%\addcontentsline{toc}{section}{References}
\bibliographystyle{amsplain}
\bibliography{ebiblio09}

\end{document}